\documentclass[12pt,reqno]{amsart}
\usepackage{soul}
\usepackage{multicol,amsmath,graphics,cancel,soul,color,bbm,amsfonts,dsfont,tikz,mathrsfs,amssymb}
\usepackage[left=1in, right=1in, top=1.1in,bottom=1.1in]{geometry}
\usetikzlibrary{matrix,patterns,positioning}
\numberwithin{equation}{section}

\newcommand{\E}{\mathbb{E}}

\newcommand{\N}{\mathbb{N}}
\newcommand{\Pb}{\mathbb{P}}

\newcommand{\R}{\mathbb{R}}

\newcommand{\vertiii}[1]{{\left\vert\kern-0.25ex\left\vert\kern-0.25ex\left\vert #1 
    \right\vert\kern-0.25ex\right\vert\kern-0.25ex\right\vert}}

\def\R{\mathbb{R}}

\def\NN{\mathbb{N}}

\def\al{\alpha}

\def\dh2l{\mathbf{d}_{\mathbb{H}_{2\ell}}}
\def\d2{\mathbf{d}_2}

\newcommand{\Indi}[1]{\mathbbm{1}_{#1}}

\newtheorem{thm}{Theorem}[section]
\newtheorem{lemma}[thm]{Lemma}

\newtheorem{theorem}{Theorem}[section]

\theoremstyle{remark}
\newtheorem{remark}[theorem]{Remark}
\theoremstyle{definition}

\newcounter{dummy} \numberwithin{dummy}{section}

\newtheorem{Proposition}[dummy]{Proposition}
\newtheorem{Theorem}[dummy]{Theorem}

\def\1{{\rm l}\hskip -0.21truecm 1}
\begin{document}

\title[Approximation of smooth numbers for harmonic samples]{Approximation of smooth numbers for harmonic samples: a Stein method approach}
\author{Arturo Jaramillo, Xiaochuan Yang}
\address{Arturo Jaramillo: Department of Probability and statistics, Centro de Investigaci\'on en matem\'aticas (CIMAT)}
\email{jagil@cimat.mx}
\address{Xiaochuan Yang: Department of Mathematics, Brunel University London}
\email{xiaochuan.yang@brunel.ac.uk}

\keywords{Probabilistic number theory, Dickman approximation, Stein method}
\date{\today}

\subjclass{60F05, 11K65}
 
\begin{abstract}
We present a de Bruijn-type approximation for quantifying the content of $m$-smooth numbers, derived from samples obtained through a probability measure over the set $\{1,\dots, n\},$ with point mass function at $k$ proportional to $1/k$. Our analysis is based on a stochastic representation of the measure of interest, utilizing weighted independent geometric random variables. This representation is analyzed through the lens of Stein's method for the Dickman distribution. A pivotal element of our arguments relies on precise estimations concerning the regularity properties of the solution to the Dickman-Stein equation for Heaviside functions, recently developed by Bhattacharjee and Schulte in \cite{BattSchulte}. Remarkably, our arguments remain mostly in the realm of probability theory, with Mertens’ first and third theorems standing as the only number theory estimations required.
\end{abstract}

\maketitle
\section{Introduction}\label{Sec:intro}
\noindent Let $\mathcal{P}$ denote the set of non-negative prime numbers and $[n]$ the integer interval $[n]:=\{1,\dots, n\}$. Define the function $\psi:\N\rightarrow\N$ through
$$\psi(k):=\max\{p\in\mathcal{P}\ ;\ p\text{ divides } k\}.$$
We say that a given natural number $k$ is $m$-smooth, with $m\in\N$, if $\psi(k)\le m$. Namely, if the prime factors of $k$ are at most $m$. The primary goal of this manuscript is to introduce a new probabilistic measurement $\Psi_H[m,n]$ for the content of the set of $m$-smooth numbers bounded by $n$, and find sharp approximations for it. Our principal contribution is an explicit estimation for $\Psi_H[m,n]$, expressed in terms of the Dickman function, in the spirit of the work developed by de Bruijn in \cite{MR0046375}.\\

\subsection{Overview on the asymptotics for the size of $m$-smooth numbers}
\noindent Let $\mathcal{Z}[n,m]$ denote the collection of $m$-smooth numbers bounded by $n$ and $\Psi(n,m)$ its cardinality 
\begin{align}\label{eq:Psidef}
\Psi(n,m):=|\mathcal{Z}[n,m]|. 
\end{align}
The set $\mathcal{Z}[n,m]$ is of fundamental importance in the field of number theory, finding its most notable application in the formulation of the Hildebrand condition for the resolution of the Riemann hypothesis. This condition, first presented in \cite{MR0804201} and succinctly outlined in Theorem \ref{thm:resolutionRiemann} below, involves a precise estimation for $\Psi(n,m)$, and serves as a sufficient criterion for validating the Riemann hypothesis.\\

\noindent Many pieces of work have been devoted to the study of the function $\Psi(n,m)$. Some of the most significant ones for purposes of our paper, have been documented in the compendium \cite{MR1265913} (see also \cite{MR2467549}), to which the reader is referred for a thorough examination of the topic. From a probabilistic number theory point of view, we can write 
\begin{align*}
\Psi(n,m)
  &=n\Pb[J_{n}\in\mathcal{Z}[n,m]]=n\Pb[\psi(J_{n})\leq m],
\end{align*}
where $J_{n}$ is a uniform random variable defined over a probability space $(\Omega,\mathcal{F},\mathbb{P})$. This perspective can be utilized as a pivot for proposing alternative ways for describing $\mathcal{Z}[m,n]$, based on modifications of the law of $J_{n}$ (see \eqref{eq:PsiHdef} below). These adjustments should be designed in such a way that the resulting problem becomes more tractable while retaining the capacity to recover information for $\Psi(n,m)$. Although we will mainly focus on providing a detailed resolution to the ``simplified problem'', we will briefly describe a methodology for transferring useful estimations to the study of $\Psi(m,n)$.\\

\noindent Our  modification on the law of $J_n$ is largely inspired by the recent paper \cite{ChJaYa}, which introduced the idea of using the harmonic distribution in $[n]$ as a tool for understanding divisibility properties of uniform samples. More precisely, we let $H_{n}$ be a collection of random variables defined in $(\Omega, \mathcal{F},\mathbb{P})$, with probability mass at $k$ proportional to $k^{-1}\Indi{[n]}(k)$. We then consider the function   
\begin{align}\label{eq:PsiHdef}
\Psi_{H}[n,m]
  &:=n\Pb[H_{n}\in\mathcal{Z}[n,m]]=n\Pb[\psi(H_{n})\leq m],
\end{align}
as measurement of content for $\mathcal{Z}[n,m]$. The advantage of dealing with the law of $H_n$ instead of $J_n$ resides in the fact that the prime multiplicities of $H_n$ are easily described in terms of independent geometric random variables, as highlighted in Proposition \ref{thm:condindeprep} below. This will allow us to simplify the task of estimating  $\Pb[\psi(H_{n})\leq m]$,  to approximating in sup-norm, the cumulative distribution of a random sum of the form
\begin{align*}
\frac{1}{\lambda_m}\sum_{p\in\mathcal{P}\cap[n]}\log(p)\xi_p,
\end{align*}
where the $\xi_p$ are independent geometric variables supported over $\N_0:=\N\cup\{0\}$, with success probability $1-1/p$. This reduction transforms the nature of the problem to a purely probabilistic one, where the recent advances in the theory of Stein's method for Dickman distributed random variables can be implemented (see Section \ref{sec:prelim} for further information regarding the Dickman distribution).\\

\noindent The aforementioned connection between the functions $\Psi_H$ and $\Psi$ stems from \cite[Lemma 3.1]{ChJaYa}, which asserts that the variable $J_{n}$ can be sharply approximated by a product of the form $Q_{n}H_n$, where $Q_{n}$, conditional on the value of $H_{n}$, is uniformly distributed in the set $(\mathcal{P}\cap[n/H_n])\cup\{1\}$. The precise statement establishes that for all $n\geq 21$,
\begin{align}\label{approxJnHnQn}
d_{TV}(J_{n}, H_nQ_n)
  &\leq 61\frac{\log\log(n)}{\log(n)},
\end{align}
where $d_{TV}$ denotes the distance in total variation between the probability distributions associated to the inputs. In particular, if it holds that for all strictly positive $\varepsilon$, the set $\hat{\mathcal{S}}_{\varepsilon}$,   defined through \eqref{def:Sepssharp}, is such that probability $\Pb[\psi(Q_{n}H_n)\leq m]$ can be sharply approximated, uniformly over $(n,m)\in\hat{\mathcal{S}}_{\varepsilon}$, then we would obtain the validity of the Hildebrand condition \cite{MR0804201} and consequently, the resolution of the Riemann hypothesis. This problem escapes the reach of our manuscript and is postponed for future research. \\

\subsection{Some questions for future research}
Although the problem of obtaining sharp estimates for $\Psi(n,m)$ will not be addressed in this paper, we would like to bring the reader's attention to some ideas that might be useful for the treatment of this particular problem. First we notice that a good approximation for $\Pb[\psi(H_{n})\leq m]$ is a fundamental first step for this task. Indeed, in virtue of \eqref{approxJnHnQn}, sharp estimations of $\Psi(n,m)$ can be obtained from their counterpart for $\Psi_H(n,m)$. In order to verify this, we write
\begin{align*}
\Pb[\psi(H_nQ_n)\leq m]
  &=\E[\Indi{\{\psi(H_n)\leq m\}}\Indi{\{Q_n\leq m\}}]
	  =\E[\Indi{\{\psi(H_n)\leq m\}}\E[\Indi{\{Q_n\leq m\}}\ |\ H_n]].
\end{align*}
Splitting the event $\{\psi(H_{n})\leq m\}$ into its components induced by the partition  
$$\{\{n/H_n\leq  m\}, \{n/H_n> m\}\},$$ 
and using the fact that for all $h\in\N$, it holds that
\begin{align*}
\Pb[Q_{n}\leq m\ |\ H_n=h]
  &=\left\{\begin{array}{ccc}
1  & \text{ if } & h>n/m\\
(\pi(m)+1)/(\pi(n/h)+1)	& \text{ if } & h\leq n/m,
\end{array}
\right.
\end{align*}
where $\pi:\R_{+}\rightarrow\R$ denotes the prime-counting function
\begin{align*}
\pi(x)
  &:=|\mathcal{P}\cap[1,x]|,
\end{align*}
we deduce the inequality  
\begin{align*}
\Pb[\psi(H_nQ_n)\leq m]
  &=\Pb[\psi(H_n)\leq m, n/m<H_n]+\E\left[\Indi{\{\psi(H_n)\leq m\}}
  \Indi{\{H_n\leq n/m\}}\frac{\pi(m)+1}{\pi(n/H_n)+1}\right],
\end{align*}
By a conditioning argument, the above identity yields
\begin{align*}
\Pb[\psi(H_nQ_n)\leq m]
  &=\Pb[\psi(H_n)\leq m]\Pb[ n/m<H_n\ |\ \psi(H_n)\leq m]\\
	&+\Pb[\psi(H_n)\leq m]\E\left[\frac{\pi(m)+1}{\pi(n/H_n)+1}\Indi{\{H_n\leq n/m\}}\ |\  \psi(H_n)\leq m\right].
\end{align*}
From the above discussion, we conclude that a sharp approximation for the distribution function of $ \psi(H_n)$ (or equivalently, for the function $\Psi_{H}$) can be used to transfer approximations of 
\begin{align}\label{eq:Conditionalexpectationgoal}
\E\left[\left(1+\frac{\pi(m)+1}{\pi(n/H_n)+1}\right)\Indi{\{H_n\leq n/m\}}\ |\  \psi(H_n)\leq m\right]
\end{align}
to approximations for the distribution of $\psi(H_nQ_n)$, and ultimately,  to $\Psi(n,m)$. Arguably, the most natural approach for carrying out an analysis for \eqref{eq:Conditionalexpectationgoal} would be to utilize the available estimations regarding the prime counting function. In particular, by \cite{MR3448979}, the inequality
\begin{align}\label{eq:Trudgian2}
|\pi(n)-\frac{n}{\log(n)}-\frac{n}{(\log n)^2}|
  &\leq \frac{184n}{(\log n)^3},
\end{align}
holds for $n\geq229$ (see also \cite[Section 3]{ChJaYa}), suggesting that an estimation of  \eqref{eq:Conditionalexpectationgoal} could be obtained through an analysis of 

\begin{align}\label{eq:quantityrequiredforRH}
\E\left[\left(1+\frac{m/\log(m)+m/\log(m)^2+1}{ n/(H_n\log(n/H_n))+n/(H_n\log(n/H_n)^2)+1}\right)\Indi{\{H_n\leq n/m\}}\ |\  \psi(H_n)\leq m\right].
\end{align}
The above term can be approximated, contingent on sharp estimates for
\begin{align*}
\E\left[g_m(n/H_n)\ |\  \psi(H_n)\leq m\right],
\end{align*}
for 
\begin{align*}
g_m(x)
  &:=\left(1+\frac{m/\log(m)+m/\log(m)^2+1}{x(\log(x)+\log(x)^2)+1}\right)\Indi{\{m\leq x\}}.
\end{align*}
Approximations on this object are closely related to the description of the limiting behavior of the measure $\mu_{n,m}$,  obtained as the distribution of $n/H_{n}$, conditional to  $\psi(H_n)\leq m$.  We conjecture that the measure $\mu_{n,m}$ concentrates around its mean and, up to a suitable normalization constant, admits an asymptotic limiting distribution. We intend to address this topic in a forthcoming paper.\\


\subsection{An overview on the description of $\Psi_{H}$}
Now we turn our attention to the function $\Psi_{H}$, which is the main focus of our investigation. Similarly to $\Psi$, the description of the asymptotic behavior of  $\Psi_{H}$ turns out to be closely related to the so-called Dickman function $\rho$, defined as the unique solution to the initial value problem of the delay type
\begin{align}\label{eq:rhodef}
u\rho^{\prime}(u)
  =-\rho(u-1),
	\quad\quad\quad\quad\quad\rho(u)=1\text{ for }u\in[0,1].
\end{align}
However, the structure of dependence on $m$ and $n$ through $\rho$ for the approximation of $\Psi_{H}(m,n)$, is different from its counterpart for $\Psi(m,n)$. In particular, our main result, presented in Theorem \ref{thm:mainone}, combined with Theorem \ref{SharpPsiHildebrand}, imply that the quotient $$(\Psi(m,n)/\Psi_{H}(m,n))/(\log(n)/\log(m))$$ 
converges to one as the quotient $\log(n)/\log(m)$ tends to infinity. In particular, $\Psi(m,n)$ is not asymptotically equivalent to $\Psi_{H}(m,n)$ as $\log(n)/\log(m)$ tends to infinity. A second remarkable feature of $\Psi_{H}$ is the fact that it can be sharply estimated uniformly over $m,n\in\N$, unlike $\Psi(m,n)$, whose approximation has shown to be tractable only for regimes of $m$ and $n$ for which $\log(n)/\log(m)\leq f(n)$, with $f(x)$ being an adequate function satisfying $f(x)\leq \sqrt{x}$.\\

\noindent The approach we follow for studying $\Psi_{H}$ is entirely probabilistic, and up to some important but tractable technical computations, consists of two fundamental steps: (i) First, we utilize a representation of $H_{n}$ in terms of independent geometric random variables conditioned on an adequate event. This reduces the analysis of $\Pb[H_{n}\in\mathcal{Z}[m,n]]$, to the estimation of the distribution $F_n$ of the sum of weighted geometric random variables satisfying that $F_{n}$ converges pointwise as $n$ tends to infinity to the Dickman function (ii) As a second step, we measure the quality of the aforementioned approximation in the sup-norm, which is done by interpreting $\rho$  as a constant multiple of the cumulative distribution function of a random variable following the so-called Dickman law and then applying the Stein-method theory for the Dickman distribution introduced in \cite{MR4003564}. For readability, the proof of our main result is presented conditional on the validity of a series of key steps, whose proof is postponed to later sections.\\

\noindent The rest of the paper is organized as follows: in Section  \ref{sec:resultsforPsi}, we review some of the work related to the estimations of $\Psi$ and their relevance in the field of number theory. In Section \ref{sec:mainrsults}, we present our main results and briefly explain the key ingredients involved in its resolution. In Section \ref{sec:prelim}, we present some preliminary tools to be utilized throughout the paper and set the appropriate notation. In Section \ref{sec:profofmainthm} we present the proof of the main theorem and finally, in sections \ref{sec:lem:simplificationtwo} and  \ref{sec:lem:kolmmain}, we prove some  technical lemmas utilized in the proof of our main result.

\section{Some known approximation results for $\Psi(n,m)$}\label{sec:resultsforPsi}
In this section, we revise some of the known results regarding the study of approximations for $\Psi(n,m)$. The material of this section is largely taken from \cite{MR1265913}.\\

\subsection{The function $\Psi$ and its relation to the Dickman function}
In the seminal paper \cite{Dickman}, Dickman presented a groundbreaking result regarding the behavior of the function $\Psi$ and its relation to the solution to the delay equation \eqref{eq:rhodef}. Dickman showed that for every $x\in\R_{+},$ the following convergence holds
\begin{align*}
\lim_{n\rightarrow\infty}\frac{1}{n}\Psi(n,n^{1/x})
  &=\rho(x).
\end{align*}
This relation was further refined by Ramaswami in \cite{MR0031958}, who proved that for any fixed $x\ge 1$, there exists $C_x>0$, such that 
\begin{align}\label{eq:convergencedickmanorig}
\left|\frac{1}{n}\Psi\left(n,n^{1/x}\right)-\rho(x)\right|
  &\leq \frac{C_x}{\log(n)}.
\end{align}
These two results where subsequently explored thoroughly by many authors, seeking for extensions that allow for uniformity over $x$. To pin down precisely the aforementioned notion of uniformity, we observe that if we set the smoothness threshold $n^{1/x}$ appearing in \eqref{eq:convergencedickmanorig}, to be equal to $m$, we have that $x=\Upsilon(n,m)$, where $\Upsilon:\N\rightarrow\R$ is given by 
\begin{align*}
\Upsilon(n,m)
  &:=\frac{\log(n)}{\log(m)},
\end{align*}
so that the Dickman approximation \eqref{eq:convergencedickmanorig}, reads $\Psi\left(n,m\right)/n\approx\rho\circ\Upsilon(n,m)$, provided that we take  $m=n^{1/x}$. Ramaswani's approximation, on the other hand,  becomes 
\begin{align*}
\left|\frac{1}{n}\Psi(n,m)-\rho\circ\Upsilon(n,m)\right|
  &\leq C_{\Upsilon(n,m)}/\log(n).
\end{align*}
One of the most significant generalizations to the above results  was done by de Bruijn in \cite{MR0046375}, who proved Theorem \ref{debruijn} below 

\begin{Theorem}[de Bruijn, 1951]\label{debruijn}
For a fixed value of $\varepsilon>0$,  define the set 
\begin{align*}
\mathcal{S}_{\varepsilon}
  &= \{(n,m)\in\N^2\ |\ n\geq 3\text{ and } \Upsilon(n,m)\leq \log(x)^{3/5-\varepsilon}\}
\end{align*}
Then, there exists a constant $C_{\varepsilon}$ only depending on $\varepsilon$, such that 
\begin{align}\label{PsiapproxBruijn}
\sup_{(m,n)\in \mathcal{S}_{\varepsilon}}\frac{1}{M_{m,n}}\left|\Psi(n,m)-n\rho\circ\Upsilon(n,m)\right|
  &\leq C_{\varepsilon},
\end{align}
where 
\begin{align*}
M_{m,n}
  &:=\frac{n\rho\circ\Upsilon(n,m)\log(\Upsilon(n,m)+1)}{\log(m)}.
\end{align*}
Otherwise said, it holds that 
\begin{align*}
\Psi(n,m)
  &= n\rho\circ\Upsilon(n,m)\left(1+O\left(\frac{\log(\Upsilon(n,m)+1)}{\log(m)}\right)\right)
\end{align*}
\end{Theorem}

The mathematical community has put many efforts in finding sharp upper bounds asymptotically equivalent to \eqref{PsiapproxBruijn}, uniformly over a larger set than $\mathcal{S}_{\varepsilon}$. Next, we mention some of the most relevant ones for the purposes of our paper. We begin with the work by Hensley in \cite{MR0814007}, where it was shown that the relaxation 
\begin{align*}
\Psi(n,m)
  &\geq n\rho\circ\Upsilon(n,m)\left(1+O\left(\frac{\log(\Upsilon(n,m)+1)}{\log(m)}\right)\right)
\end{align*}
holds uniformly over considerably larger set than $\mathcal{S}_{\varepsilon}$, thus reducing the problem of estimating the exact value of $\Psi(n,m)$, to simply finding upper bounds for it. A crucial improvement to the de Bruijn estimation was done by Hildebrand 
 in \cite{MR0831874}, where the following result was proved

\begin{Theorem}[Hildebrand, 1986]\label{SharpPsiHildebrand}
For every $\varepsilon>0$, there exits a constant $C_{\varepsilon}>0$, such that the estimation \eqref{PsiapproxBruijn} holds uniformly over $(n,m)\in\tilde{\mathcal{S}}_{\varepsilon}$, with 
\begin{align*}
\tilde{\mathcal{S}}_{\varepsilon}
  &= \{(n,m)\in\N^2\ |\ m\neq 1\text{ and } \Upsilon(n,m)\leq e^{\log(m)^{3/5-\varepsilon}}\}.
\end{align*}
\end{Theorem}
\begin{remark}
The function $e^{\log(m)^{3/5-\varepsilon}}$ appearing in the threshold is asymptotically smaller than any function of the form $y^{\alpha}$, for every $\alpha>0$. 	
\end{remark}
Shortly after, Hildebrand presents in \cite{MR0804201}, an equivalence between the resolution of this problem and the Riemann hypothesis. The precise statement reads as follows 
\begin{Theorem}[Hildebrand, 1984]\label{thm:resolutionRiemann}
For a fixed value of $\varepsilon>0$, we define the set 
\begin{align}\label{def:Sepssharp}
\hat{\mathcal{S}}_{\varepsilon}
  &:= \{(n,m)\in\N^2\ |\ m\neq 1\text{ and } \Upsilon(n,m)\leq m^{1/2-\varepsilon}\}.
\end{align}
Then, provided that exists a mapping $\varepsilon\mapsto C_{\varepsilon}$ satisfying
\begin{align}\label{eq:condforriemmequiv}
\sup_{(m,n)\in \hat{\mathcal{S}}_{\varepsilon}}\frac{\log(m)}{\log(\Upsilon(n,m)+1)}\left|\log\left(\frac{\Psi(n,m)}{n\rho\circ\Upsilon(n,m)}\right)\right|
  &\leq C_{\varepsilon},
\end{align}
then the Riemann hypothesis holds.
\end{Theorem}

\begin{remark}
The condition \eqref{eq:condforriemmequiv} can be rephrased in the ''big $O$'' notation as 
\begin{align*}
\Psi(n,m)
  &= n\rho\circ\Upsilon(n,m)\exp\left\{O\left(\frac{\log(\Upsilon(n,m)+1)}{\log(m)}\right)\right\}.
\end{align*}	
\end{remark}
Although we do not treat directly with the function $\Psi$, the above results will serve as benchmarks for setting up appropriate parallelisms for  $\Psi_{H}$.

\section{Main results}\label{sec:mainrsults}
This section describes our findings regarding the analysis of $\Psi_{H}$. Since all of the prime divisors of $H_{n}$ lie in $[n]$, then  $\psi_H(n,m)=1$ for all $m\geq n$, so the problem of estimating $\psi_H(n,m)$ can be localized into the regime $m\leq n$.\\

\noindent In the sequel, $L_{m}$ will denote the Harmonic partial sum in $[m]$ and $I_{m}$ the product of the prime numbers belonging to $[m]$, namely,
\begin{align}\label{eq:LnIndef}
L_{m}
  :=\sum_{j=1}^{m}\frac{1}{j}\ \ \ \ \ \ \ \ \ \ \ \ \  
I_{m}
  :=\prod_{p\leq m}(1-1/p).
\end{align}
The next result is the main contribution of our paper
\begin{Theorem}\label{thm:mainone}
There exists a constant $C>0$, such that for  $16\leq m$,
\begin{align}\label{eq:mainone}
\sup_{ m\leq n}\log(n)\left|\Psi_{H}(n,m)- \frac{n}{\Upsilon(n,m)}\mathcal{I}[\rho]\circ\Upsilon(n,m)\right|
  &\leq C,
\end{align}
where $\mathcal{I}$ denotes the integral operator
\begin{align*}
\mathcal{I}[\rho](x)
  &:=\int_{0}^{x}\rho(s)ds.	
\end{align*}
Otherwise said, the function $\Psi_H$ satisfies  
\begin{align*}
\Psi_{H}(n,m)
  &= \frac{n}{\Upsilon(n,m)}\mathcal{I}[\rho]\circ\Upsilon(n,m)+O(1/\log(n)),
\end{align*}
uniformly over $16\leq m\leq n$.
\end{Theorem}

A version of Theorem \ref{thm:mainone} was first addressed by Bruijn and Van Lint in \cite{MR174530} for the case where \( \sqrt{\log(n)}\leq \log(m)\leq \log(n) \). Related work followed in the contributions of Song \cite{MR1889623} and Tenenbaum \cite{MR1829871}, where \( 1/n \) was replaced by a more general weight \( h(n)/n \). While these results share some similarities with our approach, there are two fundamental differences. First, our method is entirely probabilistic, whereas all the aforementioned works rely on techniques from complex analysis. Second, the range of values for which our result is applicable is significantly broader, contrasting with the conditions in \cite{MR174530}, \cite{MR1889623}, and \cite{MR1829871}.\\

\noindent Our approach for proving \eqref{thm:mainone} is based on a representation of the law of $H_{n}$ in terms of a sequence of independent geometric random variables $\{\xi_{p}\}_{p\in\mathcal{P}}$, with $\xi_{p}$ satisfying 
\begin{align}\label{eq:deflawxip}
\Pb[\xi_p=k]
  &=(1-1/p)p^{-k},
\end{align}
for $k\geq 0$.  As explained in detail in Section \ref{sec:profofmainthm}, a stochastic representation of the multiplicities of $H_n$ in terms of the $\xi_p$,  will allow us to reduce the proof of Theorem \ref{thm:condindeprep}, to showing that the variables
\begin{align}\label{eq:Z_mdef}
Z_{m}
  &:=\sum_{p\in\mathcal{P}\cap[m]}\log(p)\xi_{p},
\end{align}
satisfy  
\begin{align}\label{eq:supnormapproxdicmanandZm}
|\Pb[\E[Z_{m}]^{-1}Z_{m}\leq z]-e^{-\gamma}\rho(z)|
  &\leq \frac{C}{\log(m)}(1+z^{-1}),
\end{align}
where $z\geq 0$ and $C$ is a universal constant independent of $m$ and $z$. This approximation is the main achievement of our proof, and its resolution will make use of the Stein method for the Dickman distribution, first introduced by Bhattacharjee and Goldstein in \cite{MR4003564} for handling approximations in the Wasserstein distance for Dickman distributed random variables, and further refined by Bhattacharjee and Schulte in  \cite{BattSchulte} for dealing with its counterpart for  the Kolmogorov metric. Equation \eqref{eq:supnormapproxdicmanandZm} resambles the findings from \cite[Theorem\ 1.1]{BattSchulte} and is of interest on its own.

\section{Preliminaries}\label{sec:prelim}
\noindent In this section, we present some important preliminaries required for the proof of Theorem \ref{thm:mainone}. There are four fundamental parts that will be utilized: (i) A methodology for assessing Kolmogorov distances for Dickman approximation, previously studied in \cite{BattSchulte,MR4003564} (ii) A stochastic representation of the multiplicities of $H_n$ (iii) Some elementary number theory estimations by Mertens and (iv) Basic formulas regarding bias transforms. \\

\subsection{Dickman distributional approximations}
In the sequel, $\gamma$ will denote the Euler-Mascheroni constant. Useful information on $\rho$ can be obtained from a probabilistic perspective by regarding the mapping $x\mapsto e^{-\gamma}\rho(x)$ as the density of a random variable. The verification of the condition $\int_{\R_{+}}e^{-\gamma}\rho(x)=1$ can be obtained from an adequate analysis of the  Laplace transform of $\rho$, as explained in \cite[Lemma 2.6.]{MR1265913} (see also \cite{MR2378655}). We say that a random variable $D$ is distributed according to the Dickman law if its cumulative distribution function is given by $e^{-\gamma}\mathcal{I}[\rho]$, where $\mathcal{I}[\rho]$ denotes the integral of $\rho$ over $[0,x]$. The main use we will give to this interpretation is the well-known characterization of the Dickman distribution in terms of its Bias transform (see \cite{MR4003564}), which states that for every bounded and smooth function $f:\R_{+}\rightarrow\R$, it holds that
\begin{align}\label{eq:biastransfchar}
\E[Df(D)]
  &=\int_0^{1}\E[f(D+t)]dt.
\end{align}
This identity can be easily obtained from the fact that if $D$ is distributed according to $e^{-\gamma}\rho$ and $U$ is a uniform random variable independent of $D$, then 
\begin{align}\label{eq:DcharactbyU}
D
  &\stackrel{Law}{=}U(D+1).
\end{align}
The reader is referred to \cite{MR3968515} for a proof of \eqref{eq:DcharactbyU}. The study of distributional approximations of random variables $X$ by means of sharp estimations on $\E[Xf(X)]$, for $f$ ranging over an adequate collection of test functions initiated with the work by Stein in \cite{MR0402873} in the study of Gaussian approximations for sums of weakly dependent random variables and has been developed for many different distributions, including exponential, beta, gamma, Poisson, geometric and Dickman, among many others (see \cite{MR2789507,MR0428387,MR3788176,MR3418541}). A very useful adaptation of this methodology to the realm of weighted sums of random variables was presented in Theorem 1.4 and Theorem 1.5 in \cite{BattSchulte}, whose proof serves as starting point for our application in hand. Replacing the role of $f$ in \eqref{eq:biastransfchar} by $f^{\prime}$, one obtains the following characterization, which can be found in \cite[Lemma 3.2]{MR4003564}
\begin{lemma}
The variable $W$ follows a Dickman law if and only if for every Lipchitz function $f$, the following identity holds
\begin{align*}
\E[Wf^{\prime}(W)-f(W + 1) + f(W)] = 0. 
\end{align*}
\end{lemma}
Let $\mathcal{C}$ denote the collection of test functions 
\begin{align*}
\mathcal{C}
  &:=\{\Indi{[0,z]}\ ;\ z\in\R_{+}\}.
\end{align*}
Taking into consideration the above lemma, we implement Stein's heuristics in the context for the Dickman law, which roughly speaking, aims to conclude an approximation of the type $\E[h(W)]\approx\E[h(D)]$ for $h\in\mathcal{C}$ by an approximation of the type 
$$\E[Wf^{\prime}(W)-f(W + 1) + f(W)]\approx 0,$$ 
for $f$ belonging to a family of test functions determined by $\mathcal{C}$. This idea can be formalized by considering the Stein equation 
\begin{align}\label{eq:steinequation}
xf_{z}^{\prime}(x)+f_{z}(x)-f_{z}(x+1)=h_{z}(x)-\E[h_{z}(D)],
\end{align}  
where $h_{z}(x):=\Indi{[0,z]}(x)$. Evaluation at $x$ equal to the variable of interest, taking expectation and then supremum with respect to $z$,  then yields a tractable expression for the distance in Kolmogorov of the variable of interest and the Dickman law.  The solution to \eqref{eq:steinequation}, along with a description of its regularity properties is presented in Theorem 1.9 and Lemma 3.3, from \cite{BattSchulte}. These results are summarized in Lemma \ref{lem:rechnicalf1} below.

\begin{lemma}\label{lem:rechnicalf1}
Let $D$ be a random variable following the Dickman law. There exists a solution $f_z$ to the equation \eqref{eq:steinequation}, admitting a decomposition of the type  $f_{z}=f_{1,z}+f_{2,z}$, where
\begin{align}\label{eq:f1def}
f_{1,z}(x)
  &:=1\wedge (z/x)-\Pb[D\leq z], 
\end{align}
and $f_{2,z}$ satisfies $f_{2,z}^{\prime}=u_{z,+}-u_{z,-}$, for functions $u_{z,+},u_{z,-}$ non-negative and non-increasing, with 
$$\|u_{z,+}\|_{\infty},\|u_{z,-}\|_{\infty}\leq 1.$$ 
Moreover, for every random variable $W\geq 0$ and constants  $t>0$, $\alpha>0$ and $u\geq 0$, 
\begin{align*}
\E\left[\frac{\alpha^{t}}{(W+u)^{1+t}}\Indi{\{W+u>\alpha\}}\right]
  &\leq \zeta(1+t)\left(C+2\frac{d_{K}(W,D)}{\alpha\vee u}\right),
\end{align*}
for some constant $C\geq 0$, where $\zeta$ denotes the Riemann zeta function.
\end{lemma}

\subsection{Mertens' formulas}
Next we present some elementary number theory estimations related to partial products and partial sums over the primes for the functions $(1-1/p)$ and $\log(p)/p$, respectively. Such results are attributed to Franz Mertens and the estimations are referred to as ``first, second and third Mertens' formulas''. The first one establishes that 
\begin{align}\label{Firstmertens}
|\sum_{p\in\mathcal{P}_{n}}\frac{\log(p)}{p}-\log(n)|
  &\leq \frac{2}{\log(n)}.
\end{align}
The second Merten's formula reads 
\begin{align}\label{Secondmertens}
\left|\sum_{p\in\mathcal{P}_n}\frac{1}{p}-\log\log(n)-c_1\right|
  &\leq 5/\log(n),
\end{align}
where $c_1$ is a universal constant, while the third guarantees the existence of a constant $C>0$, such that
\begin{align}\label{eq:Mertens}
|\log(m)I_m-e^{-\gamma}|
  &\leq C\frac{1}{\log(m)}.
\end{align}
The reader is referred to \cite{MR3363366}, pages 15-17,  for a proof of these estimations.\\

\subsection{The geometric distribution and the multiplicities of Harmonic samples}
Let $\{\xi_p\}_{p\in\mathcal{P}}$ be a sequence of independent random variables with distribution given as in \eqref{eq:deflawxip}. For $k\in\NN, p\in\mathcal P$, let $\alpha_p(k)$ denote the largest $m$ satisfying $p^m | k$. Therefore $k$ is uniquely determined by $\al_p$'s through the prime factorisation $k=\prod_{p\in\mathcal P} p^{\alpha_p(k)}$. Recall the definitions of $L_n$ and $I_n$, given by \eqref{eq:LnIndef}. The following result, first presented in \cite{ChJaYa}, describes the join distribution of the $\alpha_p(H_n)$'s.

\begin{Proposition}\label{thm:condindeprep}
 
The law of $H_{n}$ can be represented in terms of the $\xi_p$'s as
\begin{align*}
\mathcal{L}((\alpha_{p}(H_{n})\ ;\ p\in\mathcal{P}_n))
  &=\mathcal{L}((\xi_{p}\ ;\ p\in\mathcal{P}\cap[n])\ |\ A_{n}),
\end{align*}
where 
\begin{align}\label{eq:Andef}
A_{n}
  &:=\{\prod_{p\in\mathcal{P}\cap[n]}p^{\xi_{p}}\leq n\}.
\end{align}
Moreover, if $m\geq 21$, then 
\begin{align}\label{eq:Anbound}
\Pb[A_{n}]=L_nI_n\geq 1/2.	
\end{align}

\end{Proposition}
\noindent The next result gives some useful identities for exploiting the distributional properties of the $\xi_p$'s
\begin{lemma}\label{lemma:sizebiasgeom}
Let $\xi$ be a geometric random variable, with distribution 
\begin{align*}
\Pb[\xi=k]
  &=(1-\theta)\theta^{k}, \quad k\ge 0, ~\theta\in (0,1).
\end{align*}
Then, for every compactly supported function $f:\R\rightarrow\R$, 
\begin{align*}
\E[\xi f(\xi)]
  &=\frac{\theta}{1-\theta}\E[f(\xi+\tilde{\xi}+1)].
\end{align*}
where $\tilde{\xi}$ is an independent copy of $\xi.$
\end{lemma}
\begin{proof}
It suffices to consider the case $f(x)=\exp\{\textbf{i} \lambda x\}$ for some $\lambda\in\R$. For such instance, we have that 
\begin{align*}
\E[\xi f(\xi)]
  &=\frac{1}{\textbf{i}}\frac{d}{d\lambda}\E[\exp\{\textbf{i} \lambda \xi\}]
	=\frac{1}{\textbf{i}}\frac{d}{d\lambda}\frac{1-\theta}{1-\theta e^{\textbf{i}\lambda}}\\
	&= \frac{\theta (1-\theta)e^{\textbf{i}\lambda}}{(1-\theta e^{\textbf{i}\lambda})^2}
	=\frac{\theta}{1-\theta}\E[\exp\{\textbf{i}\lambda (\xi+\tilde{\xi}+1)\}].
\end{align*}
The result follows from here.
\end{proof}
%
 %

\section{Proof of Theorem \ref{thm:mainone}}\label{sec:profofmainthm}
In the sequel,  $\{H_{n}\}_{n\geq 1}$ will denote a sequence of random variables defined over $(\Omega, \mathcal{F},\mathbb{P})$, with $H_{n}$ harmonically distributed over $[n]$, namely, $\Pb[H_{n}=k]:=\frac{1}{L_{n}k}$ for $k\in[n]$ and zero otherwise, where $L_{n}$ is given by \eqref{eq:LnIndef}.\\

\noindent The proof of the bound consists on proving the following steps\\


\noindent \textbf{Step I}\\
As a first step, we will show the following result

\begin{lemma}\label{lem:simplificationtwo}
For every $m\geq 1$, it holds that 
\begin{align*}
\Pb[\psi(H_{n})\leq m]
  &=\frac{1}{L_nI_m}\Pb\left[ S_{m}\leq \frac{\log(n)}{\lambda_m}\right],
\end{align*}
where $S_{m}:=Z_m/\lambda_m$, with  $Z_{m}$ defined by \eqref{eq:Z_mdef} and $\lambda_m$ defined through
$$\lambda_{m}:=\sum_{q\in\mathcal{P}_m}\frac{\log(q)}{q}.$$ 

\end{lemma}
\noindent This way, the problem reduces to the analysis of two pieces: the estimation of the term $L_nI_m$, which can be handled by means of Merten's formula, and the analysis of the probability in the right, which now consists on the treatment of a sum of independent random variables. \\

\noindent \textbf{Step II}\\
In the second step we will show 

\begin{lemma}\label{lem:kolmmain}
There exists a universal constant $C>0$, such that
\begin{align*}
|\Pb[S_{m}\leq z]-e^{-\gamma}\mathcal{I}[\rho](z)|
  &\leq \frac{C(1+1/z^2)}{\log(m)},
\end{align*}
for some universal constant $C>0$. 
\end{lemma}

\noindent Contingent on the validity of these lemmas, we can prove Theorem \ref{thm:mainone}.

\begin{proof}[Proof of Theorem \ref{thm:mainone}]

Firstly, by Lemma \ref{lem:simplificationtwo}, 
 
\begin{multline*}
\left|\Pb[\psi(H_{n})\leq m]-\frac{1}{\log(m)\Upsilon(n,m)I_{m}}\Pb\left[ S_{m}\leq \frac{\log(n)}{\lambda_m}\right]\right|\\
\begin{aligned}
  &=\left|\frac{1}{L_nI_{m}}\Pb\left[ S_{m}\leq \frac{\log(n)}{\lambda_m}\right]-\frac{1}{\log(n)I_{m}}\Pb\left[ S_{m}\leq \frac{\log(n)}{\lambda_m}\right]\right|\\
  &=\left|1-\frac{\log(n)}{L_n}\right|\frac{1}{\log(n)I_{m}}\Pb\left[ S_{m}\leq \frac{\log(n)}{\lambda_m}\right].
\end{aligned}
\end{multline*}
This relation, together with Merten's lemma and  the approximation 
$$\log(n)\leq L_{n}\leq\log(n)+1,$$ 
leads to the inequality
\begin{align*}
\left|\Pb[\psi(H_{n})\leq m]- \frac{1}{\log(m)\Upsilon(n,m)I_{m}}\Pb\left[ S_{m}\leq \frac{\log(n)}{\lambda_m}\right]\right|
  &\leq \frac{C}{\log(n)\Upsilon(n,m)}.
\end{align*}
By Merten's approximation \eqref{eq:Mertens}, we then conclude that 
\begin{align}\label{eq:ineqmainvod2}
\left|\Pb[\psi(H_{n})\leq m]- \frac{e^{\gamma}}{\Upsilon(n,m)}\Pb\left[ S_{m}\leq \frac{\log(n)}{\lambda_m}\right]\right|
  &\leq  \frac{C}{\Upsilon(n,m)}\left(\frac{1}{\log(m)}+\frac{1}{\log(n)}\right)\leq K/\log(n),
\end{align}
Thus, by Lemma \ref{lem:kolmmain}, 
\begin{multline}\label{eq:ineqmainvod2}
\left|\Pb[\psi(H_{n})\leq m]- \frac{e^{-\gamma}}{\Upsilon(n,m)}\mathcal{I}[\rho]\left(\Upsilon(n,m)/\frac{\lambda_m}{\log(m)}\right)\right|\\
  \leq K\left(\frac{1}{\Upsilon(n,m)\log(m)}+\frac{1}{\log(m)}\right)\leq C/\log(n),
\end{multline}
for some universal constant $C>0$. By the mean value theorem and the boundedness of $\rho$, we deduce the existence of a constant $C>0$ such that 
\begin{align*}
\left|\mathcal{I}[\rho]\left(\Upsilon(n,m)/\frac{\lambda_m}{\log(m)}\right)-\mathcal{I}[\rho]\left(\Upsilon(n,m)\right)\right|
  &\leq C\Upsilon(n,m) \left|1-\frac{\lambda_m}{\log(m)}\right|.
\end{align*}
An application of Merten's approximation \eqref{Firstmertens}, then yields
\begin{align*}
\left|\mathcal{I}[\rho]\left(\Upsilon(n,m)\frac{\lambda_m}{\log(m)}\right)-\mathcal{I}[\rho]\left(\Upsilon(n,m)\right)\right|
  &\leq C\frac{1}{\log(m)}.
\end{align*}
Combining this  with inequality \eqref{eq:ineqmainvod2} then gives 
\begin{align*}
\left|\Pb[\psi(H_{n})\leq m]- \frac{e^{-\gamma}}{\Upsilon(n,m)}\mathcal{I}[\rho]\left(\Upsilon(n,m)\right)\right|
  &\leq C\frac{1}{\Upsilon(n,m)\log(m)}\leq C\frac{1}{\log(n)}.
\end{align*}
Theorem \ref{thm:mainone}  follows from here.
\end{proof}

\section{Proof of Lemma \ref{lem:simplificationtwo}}\label{sec:lem:simplificationtwo}
Observe that by Proposition \ref{thm:condindeprep}
\begin{align*}
\Pb[\psi(H_{n})\leq m]
  &=\Pb[\psi(\prod_{p\in\mathcal{P}_n}p^{\alpha_{p}(H_n)})\leq m]\\
	&=\Pb[\psi(\prod_{p\in\mathcal{P}_n}p^{\xi_{p}})\leq m\ |\ A_{n}]\\
  &=\frac{1}{\Pb[A_{n}]}\Pb[\xi_q=0\ \text{ for all }\ q\in\mathcal{P}\cap(m,n] \ \text{ and }\ \prod_{p\in\mathcal{P}_{n}}p^{\xi_{p}}\leq n],
\end{align*}
Observe that in the event $ \{\xi_q=0\ \text{ for all }\ q\in\mathcal{P}\cap(m,n]\}$, it holds that the product $\prod_{p\in\mathcal{P}_{n}}p^{\xi_{p}}$ coincides with $\prod_{p\in\mathcal{P}_{m}}p^{\xi_{p}}$. Consequently, by the independence of the $\xi_{p}$'s as well as identity \eqref{eq:Anbound}, we have that 
\begin{align*}
\Pb[\psi(H_{n})\leq m]
  &=\frac{1}{L_nI_n}\big(\prod_{q\in\mathcal{P}\cap(m,n]}(1-1/q)\big)\Pb[ \prod_{p\in\mathcal{P}_{m}}p^{\xi_{p}}\leq n]\\
  &=\frac{1}{L_n I_m}\Pb[ \sum_{p\in\mathcal{P}_{m}}\log(p)\xi_{p}\leq \log(n)].
\end{align*}
The result easily follows from here.

\section{Proof of Lemma \ref{lem:kolmmain}}\label{sec:lem:kolmmain}
\noindent Let $D$ a Dickman distributed random variable. Let $U$ be uniform in $(0,1)$ independent of $\{\xi_p\}_{p\in\mathcal{P}}$.
Let $f_{z}$ be the solution to the Stein equation \eqref{eq:steinequation}. For convenience in the notation, we will simply write $f$ instead of $f_{z}$.  Define 
\begin{align*}
S_{m}^{q}
  &:=S_{m}-\frac{\log(q)}{\lambda_{m}}\xi_q, \quad q\in \mathcal P_m,
\end{align*}
with 
$$\lambda_{m}:=\sum_{q\in\mathcal{P}_m}\frac{\log(q)}{q},$$ 
and observe that 
\begin{align*}
\Pb[S_{m}\leq z]-\Pb[D\leq z]
  &=\E[S_{m}f^{\prime}(S_{m})-f^{\prime}(S_{m}+U)]\\
	&=\frac{1}{\lambda_{m}}\sum_{q\in\mathcal{P}\cap[m]}\log(q)\E[\xi_qf^{\prime}(S_{m}^q+\frac{\log(q)}{\lambda_m}\xi_q)]-\E[f^{\prime}(S_{m}+U)].
\end{align*}
By Lemma \ref{lemma:sizebiasgeom}, 
\begin{multline*}
\Pb[S_{m}\leq z]-\Pb[D\leq z]\\
	=\frac{1}{\lambda_{m}}\sum_{q\in\mathcal{P}\cap[m]}\frac{\log(q)}{q}(1-1/q)^{-1}\E[f^{\prime}(S_{m}+\frac{\log(q)}{\lambda_{m}}\tilde{\xi}_q+\frac{\log(q)}{\lambda_{m}})]-\E[f^{\prime}(S_{m}+U)],
\end{multline*}
where $\{\tilde{\xi}_p\}_{p\in\mathcal{P}}$ is an independent copy of $\{\xi_p\}_{p\in\mathcal{P}}$. By conditioning the expectation in the right hand side on the value of $\tilde{\xi}_q$, we obtain 
\begin{align*}
\Pb[S_{m}\leq z]-\Pb[D\leq z]
	&=T_{1}+T_2,
\end{align*}
where 
\begin{align*}
T_1
	&:=\frac{1}{\lambda_{m}}\sum_{q\in\mathcal{P}\cap[m]}\frac{\log(q)}{q}\E[f^{\prime}(S_{m}+\frac{\log(q)}{\lambda_{m}})]-\E[f^{\prime}(S_{m}+U)]\\
T_{2}
	&:=\frac{1}{\lambda_{m}}\sum_{q\in\mathcal{P}\cap[m]}\frac{\log(q)}{q^2}(1-1/q)^{-1}\E[f^{\prime}(S_{m}+\frac{\log(q)}{\lambda_{m}}\tilde{\xi}_q+\frac{\log(q)}{\lambda_{m}})\ |\ \tilde{\xi}_q\geq 1].
\end{align*}
Let $V_{m}$ be a random variable independent of the $\xi_{p}$'s, with distribution given by 
\begin{align}\label{Vmdef}
\Pb[V_{m}=\log(q)/\lambda_m]
  &=\frac{\log(q)}{q\lambda_{m}},
\end{align} 
for $q\in\mathcal{P}\cap[m]$ and zero otherwise. With this notation in hand, we can write
\begin{align*}
T_1
  &=\E[f^{\prime}(S_{m}+V_{m})]-\E[f^{\prime}(S_{m}+U)].
\end{align*}

\noindent By Lemma \ref{lem:rechnicalf1}, $f^{\prime}:=f_1^{\prime}+u_{+}-u_{-}$, where $f_{1}$ is given by \eqref{eq:f1def} and $u_{+},u_{-}$ are non-negative and non-increasing, with $\|u_{+}\|_{\infty},\|u_{-}\|_{\infty}\leq 1$. We can thus write 
\begin{align*}
|T_1|
  &\leq |T_{1,1}|+|T_{1,2}|+ |T_{1,3}|,
\end{align*}
where 
\begin{align}\label{eq:Tisdef}
T_{1,1}
  &:=\E[f_1^{\prime}(S_{m}+V_{m})]-\E[f_1^{\prime}(S_{m}+U)]\nonumber\\
T_{1,2}
  &:=\E[u_{+}(S_{m}+V_{m})]-\E[u_{+}(S_{m}+U)]\nonumber\\
T_{1,3}
  &:=\E[u_{-}(S_{m}+V_{m})]-\E[u_{-}(S_{m}+U)].
\end{align}
Let $F_{m}$ be the distribution function of $V_{m}$ and $F_{m}^{-1}$ its left inverse. Since $F_{m}^{-1}(U)$ is distributed according to $V_{m}$, it follows that $(F_{m}^{-1}(U),U)$ is a coupling for $V_{m}$ and $U$, which allows us to write $F_{m}^{-1}(U)$ instead of $V_{m}$ in \eqref{eq:Tisdef}. Observe that the function $F_{m}^{-1}$ remains constant equal to $\log(p)/\lambda_m$ in the intervals $I_p=[a_p,b_p)$, where
\begin{align}\label{eq:aandbdef}
a_p:=\sum_{q\in\mathcal{P}\cap[p]}\frac{\log(q)}{q\lambda_{m}}\ \ \ \ \ \text{ and }\ \ \ \ \ b_p:=\sum_{p\in\mathcal{P}\cap[p]}\frac{\log(q)}{q\lambda_{m}}+\frac{\log(p_+)}{p_+\lambda_{m}},
\end{align}
for $x_+$, defined as the smallest prime strictly bigger than $x$. Observe that $a_p$ and $b_p$ also depend on $t$ and $m$, but we have avoided this in the notation  for convenience. By the Mertens' approximation \eqref{Firstmertens}, there exists a constant $C>0$ such that 
\begin{align}\label{ineq:mertenforab}
\left|a_{p}-\frac{\log(p)}{\lambda_m}\right|,\left|b_{p}-\frac{\log(p)}{\lambda_m}\right|
  &\leq \frac{C}{\log(p)\log(m)}.
\end{align}
Localizing $V_{m}$ and $U$ over the intervals of constancy of $F^{-1}_m$, we can write the terms $T_{1,i}$ as 
\begin{align*}
T_{1,1}
	&=\sum_{p\in\mathcal{P}\cap[m]}\E[(f_1^{\prime}(S_{m}+\log(p)/\lambda_{m})-f_1^{\prime}(S_{m}+U))\Indi{I_p}(U)]\\
T_{1,2}
	&=\sum_{p\in\mathcal{P}\cap[m]}\E[(u_{+}(S_{m}+\log(p)/\lambda_{m})-u_{+}(S_{m}+U))\Indi{I_p}(U)]\\
T_{1,3}
	&=\sum_{p\in\mathcal{P}\cap[m]}\E[(u_{-}(S_{m}+\log(p)/\lambda_{m})-u_{-}(S_{m}+U))\Indi{I_p}(U)].
\end{align*}
In order to handle the term $T_{1,1}$, we proceed as follows. Let $\psi_{p,m}$ denote the term 
\begin{align*}
\psi_{p,m}
  &:=\E[(f_1^{\prime}(S_{m}+\log(p)/\lambda_{m})-f_1^{\prime}(S_{m}+U))\Indi{I_p}(U)],
\end{align*}
so that $T_{1,1}$ can be written as
\begin{align*}
T_{1,1}
	&=\sum_{p\in\mathcal{P}\cap[m]}\psi_{p,m}.	
\end{align*}
Since $f_{1}^{\prime}$ is supported in $[z,\infty )$ and  $f_{1}(x)=zx^{-1}$ for $x\geq z$, by distinguishing the instances for which $S_{m}+\log(p)/\lambda_{m}$ is bigger than or smaller than $z$ or $U$, we can 
write $\psi_{p,m}=\theta_{p,m}+\vartheta_{p,m}^1+\vartheta_{p,m}^2$, where
\begin{align*}
\theta_{p,m}
  &:=\E[(f_1^{\prime}(S_{m}+\log(p)/\lambda_{m})-f_1^{\prime}(S_{m}+U))\Indi{I_p}(U)\Indi{\{z\leq S_{m}+\log	(p)/\lambda_{m},S_{m}+U\}})]\\
\vartheta_{p,m}^1
  &:=-\E[f_1^{\prime}(S_{m}+U)\Indi{I_p}(U)\Indi{\{S_{m}+\log(p)/\lambda_{m}< z\leq S_{m}+U\}}]\\
\vartheta_{p,m}^2
  &:=\E[f_1^{\prime}(S_{m}+\log(p)/\lambda_{m})\Indi{I_p}(U)\Indi{\{S_{m}+U< z\leq S_{m}+\log	(p)/\lambda_m\}}].
\end{align*}
By \eqref{ineq:mertenforab}, there exists $C>0$ such that  
\begin{align*}
\frac{\log p}{\lambda_m} - \frac{C}{\lambda_m\log(p)}\le  a_p \le b_p \le \frac{\log p}{\lambda_m} + \frac{C}{\lambda_m\log(p)}.
\end{align*}
Consequently, in the event $\{U\in I_p\}$, we have that
\begin{multline*}
\{S_{m}+\log(p)/\lambda_{m}< z\leq S_{m}+U\}\cup\{S_{m}+U< z\leq S_{m}+\log(p)/\lambda_{m}\}\\
  \subset 
	\{z-S_m-C/(\lambda_m\log(p)) \leq \log(p)/\lambda_m\leq z-S_m+C/(\lambda_m\log(p)) \}.
\end{multline*}
Using the fact that $f_{1}^{\prime}(x)$ satisfies $f_{1}^{\prime}(x)=-zx^{-2}$ and $|f_{1}^{\prime}(x)|\leq z^{-1}$ for $x\geq z$, we can easily show that
\begin{align*}
|\vartheta_{p,m}^2|,|\vartheta_{p,m}^1|
  &\leq z^{-1}\mathbb{P}[z-S_m-C/\lambda_m\leq \frac{\log p}{\lambda_m}\leq z-S_m+C/\lambda_m]\mathbb{P}[U\in I_p].
\end{align*}
We will assume without loss of generality that $C\geq $1. From here we deduce that 
\begin{align*}
T_{1,1}
  &\leq \sum_{p\in\mathcal{P}\cap[m]}|\theta_{p,m}|
	+ z^{-1}\E\left[\sum_{p\in\mathcal{P}\cap[e^{-C+(z-S_m)\lambda_{m}},e^{C+(z-S_m)\lambda_{m}}]}\frac{\log(p)}{p\log(m)}\right].
\end{align*} 
Splitting the sum in the right-hand side into the regimes $p\leq C$ and $p>C$, we deduce that 
\begin{align*}
T_{1,1}
  &\leq \sum_{p\in\mathcal{P}\cap[m]}|\theta_{p,m}|
	+ z^{-1}\sup_{y\in\R_{+}}\sum_{p\in\mathcal{P}\cap[C\vee e^{-C+y},C\vee e^{C+y}]}\frac{\log(p)}{p\log(m)}
 +  z^{-1} \sum_{p\in\mathcal{P}\cap[1,C]}\frac{\log(p)}{p\log(m)}.
\end{align*}
Mertens' approximation \eqref{Firstmertens}, yields the existence of a possibly different constant $K>0$, such that 
\begin{align*}
T_{1,1}
  &\leq K(1+z^{-1})/\log(m)+\sum_{p\in\mathcal{P}\cap[m]}|\theta_{p,m}|.
\end{align*}
For handling the term $\theta_{p,m}$, we use the mean value theorem, as well as the fact that $f_{1}^{\prime}(x)=-zx^{-2}$ for $x\geq z$ to write 
\begin{align*}
\theta_{p,m}
  &\leq C\frac{1}{\log(p)\log(m)}\E[z(S_{m}+\frac{\log(p)}{\lambda_{m}})^{-3}\Indi{\{z\leq S_{m}+\log(p)/\lambda_{m}\}}]\Pb[U\in I_p]\\
	&\leq z^{-2}\frac{C}{p_{+}(\log(m))^2}.
\end{align*}
Bounding $p_{+}$ by $p$ from below, adding over $p\in\mathcal{P}$ and using Mertens' second formula \eqref{Secondmertens}, we deduce the existence of a constant $C>0$ such that 
\begin{align*}
T_{1,1}
  &\leq C(1+1/z^2)\frac{1}{\log(m)}.
\end{align*}
For handling the term $T_{1,2}$, we use the non-increasing property of $u_{+}$ to write 

\begin{align*}
|T_{1,2}|
	\leq\sum_{p\in\mathcal{P}\cap[m]}\frac{\log(p)}{p\lambda_{m}}\E[(u_{+}(S_{m}+\log(p)/\lambda_{m})-u_{+}(S_{m}+\log(p)/\lambda_{m}+\log(p_+)/(p_+\lambda_{m})))].
\end{align*}
Applying summation by  parts, $u_+\ge 0$ and $\|u_+\|_\infty \le 1$, we thus get 
\begin{align*}
T_{1,2}
	\leq \frac{\log(p_{\pi(m)})}{p_{\pi(m)}\lambda_{m}}+\frac{1}{\lambda_{m}}\sum_{p\in\mathcal{P}_{m}}(\frac{\log(p)}{p}-\frac{\log(p_{+})}{p_{+}}).
\end{align*}
The second sum is telescopic, with boundary terms bounded by a constant multiple of $\lambda_{m}^{-1}$, which gives the existence of a constant $C>0$ such that 
\begin{align*}
T_{1,2}
  &\leq \frac{C}{\lambda_{m}}.
\end{align*}
The term $T_{1,3}$ can be handled in a similar fashion.\\

\noindent In order to handle the term $T_{2}$ we proceed as follows. Using the decomposition $f^{\prime}=f_{1}^{\prime}+u_{+}-u_{-}$, we obtain the  bound
\begin{align*}
T_{2}
	&\leq \frac{1}{\lambda_{m}}\sum_{q\in\mathcal{P}\cap[m]}\frac{\log(q)}{q^2}(1-1/q)^{-1}\E[f_1^{\prime}(S_{m}+\frac{\log(q)}{\lambda_{m}}\tilde{\xi}_q+\frac{\log(q)}{\lambda_{m}})\ |\ \tilde{\xi}_q\geq 1]
	+\frac{\pi^2}{3\lambda_{m}}.
\end{align*}
The conditional expectation in the right can be estimated by conditioning over $\tilde{\xi}_q$ and applying Lemma \ref{lem:rechnicalf1}. Indeed, using Lemma \ref{lem:rechnicalf1} with $W=S_m$, $u=\log(q)/\lambda_m\tilde{\xi}_q+\log(q)/\lambda_m$,  $\alpha=z$ and $t=1$, we obtain 
\begin{multline*}
\E\left[f_1^{\prime}\left(S_{m}+\frac{\log(q)}{\lambda_{m}}\tilde{\xi}_q+\frac{\log(q)}{\lambda_{m}}\right)\ |\ \tilde{\xi}_q\right]\\
\begin{aligned}
    &=z\E\left[\left(S_{m}+\frac{\log(q)}{\lambda_{m}}\tilde{\xi}_q+\frac{\log(q)}{\lambda_{m}}\right)^{-2}\Indi{\{S_{m}+\frac{\log(q)}{\lambda_{m}}\tilde{\xi}_q+\frac{\log(q)}{\lambda_{m}}> z\}}\ |\ \tilde{\xi}_q\right]\\
    &\leq \zeta(2)(C+2d_{K}(S_m,D)((\tilde{\xi}_q+1)^{-1}\lambda_m/\log q)).
\end{aligned}
\end{multline*}
Observe that 
\begin{align*}
\E[(\tilde{\xi}_q+1)^{-1}]
  &=\sum_{k=1}^{\infty}q(1-1/q)q^{-k}/k)=-q(1-1/q)\log(1-1/q)
\end{align*}

which in turn implies the relation
\begin{align*}
\E[f_1^{\prime}(S_{m}+\frac{\log(q)}{\lambda_{m}}(\tilde{\xi}_q+1) |\ \tilde{\xi}_q\geq 1]
  &\leq C-2\zeta(2)q(1-1/q)\lambda_{m}\log(1-1/q)\frac{ d_K(S_{m},D)}{\log(q)}.
\end{align*}
Therefore, for a possibly different constant $C>0$, the following inequality holds
\begin{align*}
T_2
  &\le C/\log(m) + \frac{\pi^2}{3}d_K(S_m,D)\sum_{q\in\mathcal{P}}\frac{(1-1/q)\log(1-1/q)}{q^2}\\
  &\le C/\log(m) + \frac{\pi^2}{15}d_K(S_m,D)	
\end{align*}

guaranteeing the existence of a constant $\delta\in(0,1)$, such that 

\begin{align*}
T_{2}
	&\leq d_K(S_{m},D)(1-\delta)+\frac{C}{\log(m)}.
\end{align*}
By the above analysis,
\begin{align*}
d_K(S_{m},D)
  &\leq C(1+z^{-1})/\log(m)+ (1-\delta)d_K(S_{m},D).
\end{align*}
Relation  \eqref{Firstmertens} follows from here.\\

\noindent\textbf{Acknowledgements}\\
We thank Louis Chen, Chinmoy Bhattacharjee, 
Ofir Gorodetsky and Larry Goldstein for helpful guidance in the elaboration of this paper. Arturo Jaramillo Gil was supported by the grant CBF2023-2024-2088. 

\bibliographystyle{plain}
\bibliography{bibibi}

\end{document}